\documentclass[12pt,makeidx]{amsart}
\usepackage{amssymb}

\usepackage{enumerate,color}
\usepackage[pagebackref,colorlinks,linkcolor=red,citecolor=blue,urlcolor=blue,hypertexnames=true]{hyperref}
\usepackage{url}

\renewcommand{\subset}{\subseteq}
\numberwithin{equation}{section}

\newtheorem{theorem}            {Theorem}[section]
\newtheorem{corollary}          [theorem]{Corollary}
\newtheorem{proposition}        [theorem]{Proposition}

\newtheorem{lemma}              [theorem]{Lemma}
\newtheorem{remark}         [theorem]{Remark}

\def\beq{\begin{equation}}
\def\eeq{\end{equation}}

\def\ben{\begin{enumerate}}
\def\een{\end{enumerate}}

\newcommand{\fgg}{x}

\def\cD{ {\mathcal D} }

\def\cF{ {\mathcal F} }
\def\cG{ {\mathcal G} }

\def\cS{{\mathcal S} }

\def\cW{{\mathcal W}}

\def\tA{\tilde{A}}

\def\tC{\tilde{C}}
\def\tD{\tilde{D}}
\def\tJ{\tilde{J}}

\def\RR{{\mathbb R}}

\def\hh{\hat h}

\def\hk{{\hat k}}
\def\hX{\hat{X}}

\def\smatng{\mathbb S_n(\mathbb R^g)}

\def\bem{\begin{pmatrix}}
\def\eem{\end{pmatrix}}

\def\tr{{ \tilde r}}
\def\rr{{\rho}}

\def\q0{{Q_0}}

\def\fP{\mathfrak P}
\def\hX{\mathbb X}
\def\thX{\tilde{\hX}}
\def\bM{\mathbb M}

\def\smatMg{\mathbb S_M(\mathbb R^g)}
\def\smatMNg{\mathbb S_{M,N}(\mathbb R^g)}

\def\smatdg{\mathbb S_d(\mathbb R^g)}
\def\smatmg{\mathbb S_m(\mathbb R^g)}

\def\domfor{\mbox{Domfor}}
\def\domalg{\mbox{Domalg}}
\def\dome{\mbox{Domlim}}
\def\mfI{\mathfrak{I}}
\def\bA{\mathbb A}

\def\Rx{\mathbb R\langle x\rangle}
\def\tth{{\tt{h}}}
\def\ttk{{\tt{k}}}

\def\mfr{\mathfrak{r}}
\def\mfP{\mathfrak{P}}

\def\bV{\mathbb V}

\newcommand{\df}[1]{{\bf{#1}}{\index{#1}}}

\makeindex

\begin{document}

\title[LMIs and free rationally defined convex sets]{Free Convex Sets Defined by Rational Expressions Have LMI Representations}

\author[Helton]{J. William Helton${}^1$}
\address{J. William Helton, Department of Mathematics\\
  University of California \\
  San Diego}
\email{helton@math.ucsd.edu}
\thanks{${}^1$Research supported by NSF grants
DMS-0700758, DMS-0757212, and the Ford Motor Co.}

\author[McCullough]{Scott McCullough${}^2$}
\address{Scott McCullough, Department of Mathematics\\
  University of Florida, Gainesville 
   }
   \email{sam@math.ufl.edu}
\thanks{${}^2$Research supported by NSF grant DMS-1101137.}


\date{\today}

\maketitle

\begin{abstract}
  Suppose  $p$ is a symmetric matrix 
 whose entries are  polynomials in 
  freely noncommutative  variables
  and $p(0)$ is positive definite. 
  Let $\cD_p$  denote the component of zero of the set of those $g$-tuples
 $X=(X_1,\dots,X_g)$ of symmetric matrices (of the same size)
  such that $p(X)$ is positive definite. 
  In \cite{HM} it was shown that  if $\cD_p$ is convex and bounded, then
  $\cD_p$ can be described as the set of solutions
  of  a linear matrix inequality (LMI).  
  This article extends that result from matrices of polynomials
  to matrices of rational functions in free variables.

   As a refinement of a theorem of 
  Kaliuzhnyi-Verbovetskyi and Vinnikov, 
  it is also shown that a minimal symmetric descriptor realization $r$
  for a symmetric free matrix-valued rational function $\mathfrak{r}$ 
  in $g$ freely noncommuting variables $x=(x_1,\dots,x_g)$ precisely encodes
  the singularities of the rational function.  
  This singularities result is an important  ingredient 
  in the proof of  the LMI
  representation theorem stated above.

\end{abstract}

\section{Introduction}
  Given positive integers $g$ 
  \index{$g$} \index{$n$} 
 and  $n$, let $x=(x_1,\dots,x_g)$ denote $g$ freely noncommuting
  indeterminates and \index{$\smatng$}
   let $\smatng$ denote 
  the $g$-tuples $X=(X_1,\dots,X_g)$ of symmetric $n\times n$
  matrices (with real entries). 
    Given a positive integer $d$ \index{$d$}
  and a tuple $A\in\smatdg$, let $L_A(x)$ denote the
  \df{homogeneous linear pencil} \index{$L_A$}
\[
  L_A(x) = \sum_{j=1}^g A_j x_j.
\] 
  The $d\times d$ matrix $J$ is a \df{symmetry} \index{$J$}
  if $J=J^T$  and $J^2=I$. The expression 
\[
 J-L_A(x)
\] 
 is an example of an \df{affine linear pencil}. Such a pencil 
 is naturally evaluated at $X\in\smatng$ as
\[
  J-L_A(X) = J\otimes I_n - \sum_{j=1}^g A_j\otimes X_j,
\]
  with the result a $dn\times dn$
 symmetric matrix.  Here $\otimes$ denotes the Kronecker product of
 matrices and $I_n$  the $n\times n$ identity matrix. 

  For a symmetric matrix $Y$, the notation $Y\succ 0$ (resp.
  $Y\succeq 0$) indicates that $Y$ is positive definite (resp.
  positive semidefinite). 
  In the case that $J=I$, the pencil $I-L_A$ is
 a \df{monic affine linear pencil} and, for each $n$, the set
\[
  \cD_{I-L_A}(n)= \{X\in\smatng: I-L_A(X)\succ 0\}
\]
 is convex.  The sequence of sets $(\cD_{I-L_A}(n))_n$ is
 a \df{(free) LMI domain}, synonymously a \df{(free) spectrahedron}.

  The main results of this article gives an LMI
  representation for a free bounded convex set
  determined by a free rational function. 
  Before giving a precise statement, we pause to informally describe a sample result. 
  Let $\mfr$ be an $\ell\times \ell$ matrix-valued  symmetric
 free rational function which is positive definite at the origin. 
 {\it If
 the sequence of sets
\[
 \mfP_{\mfr}(n) = \{X\in\smatng : \mfr(X) \mbox{ is defined and } \mfr(X)\succ 0\}
\]
 is uniformly bounded and each $\mfP_{\mfr}(n)$ is convex, then there is a monic
 affine linear pencil $I-L_A$ such that, for each $n$,
\[
  \mfP_{\mfr}(n) = \cD_{I-L_A(x)}(n).
\]}

  The proof depends upon a description of the domain and singularities of a free rational function
  together with the main result of \cite{HM}.   In the remainder of this introduction
  the fairly minimal necessary background on free rational functions  - in terms of descriptor realizations - 
  and their singularities needed to state the main results is developed.  Descriptor realizations and their domains are 
  the topic of Subsections \ref{sec:descriptor-realizations} and \ref{sec:domains} respectively. 
  A precise statement
  of the main results on LMI representations of free bounded convex sets
  determined by free rational functions is given in Subsection \ref{sec:LMIreps} as Theorem \ref{cor:main}.
  The proof of Theorem \ref{cor:main} hinges upon the nature of the singularities
  of descriptor realizations a discussion of which appears in Subsection \ref{sec:introhidden}.
  Subsection \ref{sec:guide} is a reader's guide to the body
  of the paper. Acknowledgments appear in Section \ref{sec:ack}.

\subsection{Descriptor realizations}
  \label{sec:descriptor-realizations}
 For a positive integer $n$, the set
\begin{equation}
 \label{eq:defPL}
   \mfI_{J-L_A(x)}(n) = \{X\in\smatng: J-L_A(X) \mbox{ is invertible} \}
\end{equation}
 is open and dense in $\smatng$. 
 The sequence $\mfI_{J-L_A(x)} = ( \mfI_{J-L_A(x)}(n))_n$ is
  \index{$\mfI$}
 the \df{invertibility set} of $J-L_A(x)$. \index{$M_{d\times\ell}$}
 Let $M_{d\times \ell}$ denote the $d\times \ell$ matrices with real entries.
 Given $C\in M_{d\times \ell}$ and $D\in M_{\ell \times \ell}$ \index{$r$}
 the expression 
\begin{equation}
 \label{eq:dreal}
 r(x) = D+ C^T(J-L_A(x))^{-1}C
\end{equation}
  is known
 as a \df{descriptor realization}.\footnote
 {More precisely $r$ is a \df{symmetric} descriptor realization, but
  we often drop the adjective symmetric here.}
 \footnote{The reason for this
 terminology is explained later.}  
  It is an example of a \df{rational expression}.  We assume throughout that
  that $r$ is \df{positive at the origin}; i.e.,
\[
  D+C^TJC \succ 0.
\]

 The rational expression
 $r$ can be evaluated at any  $X\in\mfI_{J-L_A(x)}$ as
\[
 r(X) = D\otimes I + C^T\otimes I(J-L_A(X))^{-1} C\otimes I.
\]
  In the case that $X$ has size $n$  so that $X\in\mfI_{J-L_A(x)}(n),$
   the matrix $I$ is the $n\times n$ identity and 
  the matrix $r(X)$ is symmetric  of size $\ell n\times \ell n$.
   Further, for  $X=\mathbf{0}$ - the $g$-tuple of $n\times n$ zero matrices -
 and $0\in\mathbb R^g$ the $0$ vector, 
\[
  r(\mathbf{0}) = r(0)\otimes I_n = (D+C^TJC)\otimes I_n \succ 0.
\]

\subsection{The domains of a descriptor realization}
 \label{sec:domains}
%
  It could happen that there is an $n$ and $\chi \notin \mfI_{J-L_A(x)}(n)$ 
such that
\[
 r(\chi ):= \lim_{X\to \chi} r(X) 
\]
 exists, where the limit is taken through $\mfI_{J-L_A(x)}(n).$ 
 Such a $\chi$ is called a \df{hidden singularity} of
 the rational expression  $r.$ 
 For positive integers $n$, let 
 $$\dome(r,n)= ( \mfI_{J-L_A(x)}(n)
  \cup \mbox{hidden singularities of} \ r )\ \cap \ \smatng.$$
 The \df{limit domain of $r$},
 denoted $\dome(r,n),$ \index{domlim} is the sequence $(\dome(r,n))_n$.

 The notion of a rational expression (analytic at $0$)\footnote{In this article, rational expression
  means rational expression analytic at $0.$} 
  is defined in Section \ref{sec:rat}.
  For now we note that a descriptor realization is an example and
  that rational expressions have, for each $n$, a domain in
  $\smatng$ which is both open and dense and contains a neighborhood
  of $0$. 
  Two rational expressions $r$ and $\hat{r}$ may have the property
$r(X)=\hat{r}(X)$ for all tuples of matrices for which  they 
 are both defined,
in which case we shall say they represent the same  
 noncommutative (free)  rational function $\mathfrak{r}$. \index{$\mathfrak{r}$} Thus,
 a free rational function $\mathfrak{r}$ is an equivalence class of 
  rational expressions and it is natural to define 
  another notion of domain for the rational expression $r$.   
  The \df{algebraic domain} of $r$
 is the union of the formal domains 
 of all the
 rational expressions equivalent to $r$.
 A more precise presentation, complete with 
 the definition of formal domain,  appears in Section \ref{sec:rat}.
  In any event, the algebraic domain of 
 the descriptor realization $r$ is also \index{domalg}
 a  sequence $\domalg(r)= (\domalg(r,n))_n$ and 
\[
   \domalg(r,n) \subset \dome(r,n)\subset \smatng.
\]

\begin{remark}\rm
 \label{rem:secondkind}
 For a classical perspective on the distinction between 
  limit and algebraic domains,  consider 
 the rational function
\[
   R(x,y) =\frac{x^2y^2}{x^2+y^2}.
\]
  The origin is in its limit domain, but not in its algebraic domain. 

  Generally, suppose $P$ and $Q$ are 
  classical commutative relatively prime polynomials
  in $g$ variables. A  
  {\it non-essential singularity of the second kind} of the rational
  function $R=\frac{P}{Q}$ 
  is a point where both $P$ and $Q$  vanish.
  There is a large literature in multivariable systems theory of
  (classical commutative) rational functions (transfer functions)  
  with essential singularities of the second kind 
  at points on the distinguished boundary of the polydisc.
 \end{remark}


\subsection{LMI representations}
 \label{sec:LMIreps}
 For positive integers $n$, let \index{$\fP_r$}
 \[
  \fP_r(n) = \{X \in \dome(r,n): r(X)\succ 0\}.
 \]
 The sequence $\fP_r = (\fP_r(n))$ is the \df{positivity set of $r$}.

 Similarly, let $\cD_r(n)$ \index{$\cD_r$}
  denote the component of zero of the set 
\[
  \{X\in\domalg(r,n): r(X)\succ 0 \};
\]
that is, $\cD_r(n)$ is the principal component
 of the the set of $X\in\smatng$ in the algebraic
 domain of $r$ for which $r(X)$ is positive definite.
 The assumption that $r(0)\succ 0$ implies
 the sets $\cD_r(n)$ and $\fP_r(n)$ are not empty (for each $n$).

A sequence $\cD=(\cD(n))$ of sets $\cD(n)\subset \smatng$ is 
\df{bounded} if there exists an $R\in\RR$ such that 
$X_1^2 + \dots + X_g^2 \preceq R I_n$ for each $n$
 and $X\in\cD(n)$.  The following theorem contains the main
  results of this article. 
  The notion of a minimal descriptor realization 
  is defined at the outset of Section \ref{sec:words}.
  For now we note that the condition is natural and
  that minimality can be assumed without loss of generality.

\begin{theorem}
 \label{cor:main}
   Suppose $r$ is a minimal symmetric descriptor realization with $r(0)\succ 0$. 
\begin{itemize}
 \item[(lim)] If $\fP_r$ is bounded and each $\fP_r(n)$ is convex, then 
   there exists a positive integer $m$ and a tuple $\bA\in\mathbb S_m(\mathbb R^g)$ 
  such that $X\in\fP_r(n)$ if and only if
\[
  I-L_{\bA}(X) \succ 0.
\]
 \item[(alg)]
   If $\cD_r$ is bounded and for each positive integer $n$ the
   set $\cD_r(n)$ is convex, then 
  there exists a positive integer $m$ and a tuple $\bA\in\mathbb S_m(\mathbb R^g)$ 
  such that $X\in\cD_r(n)$ if and only if
\[
  I-L_{\bA}(X) \succ 0.
\]
\end{itemize}  
\end{theorem}

Theorem \ref{cor:main} is a consequence of Theorem \ref{cor:canthide} below which asserts,
 under natural hypotheses, the absence of hidden singularities
 and the main result
 of \cite{HM}. The details are in Section \ref{sec:repCor}.

\begin{remark}\rm
 In the case that $\fP_r$ (or $\cD_r$) is convex, it is in fact \df{matrix convex}. 
 From the general theory of matrix convex sets 
 such a set can be separated from an outlier by a linear matrix inequality
 \cite{EW} (see also \cite{F04,F92,K,W,WW}). 
 The content of Theorem \ref{cor:main} is that in the case the matrix convex set
  is described as the positivity set of a rational function, then a single linear
  matrix inequality simultaneously separates all outliers from the set. 

  The theory of matrix convex sets falls squarely in the realm of 
  operator algebras and spaces \cite{BL,vernbook}.
\end{remark}

\subsection{Hidden singularities}
\label{sec:introhidden}
 There is a stronger notion
 of hidden singularity which reduces to that of hidden
 singularity under various natural hypotheses.
 A hidden singularity $\chi\in\smatng$ of $r$ is 
 \df{well hidden} if for each $m$ and 
 $K \in {\mathbb S_{m}(\mathbb R^g)}$ there
 is a $\rho_0>0$ so that, for all $|\rho|<\rho_0$, 
\[
  \chi\oplus \rho K = \begin{pmatrix} \chi& 0 \\ 0 & \rho K \end{pmatrix} \in \mathbb S_{n+m}(\mathbb R^g)
\]
 is also a hidden singularity.

 The following is our main result on well hidden singularities.
 It is a stepping stone to proving the absence of
  hidden singularities in the presence of additional mild hypotheses
 as given in 
 Theorem  \ref{cor:canthide}.

\begin{theorem}
 \label{thm:main}
   If $r$ is a minimal symmetric descriptor realization
  (with $r(0)\succ 0$), then $r$ has no well hidden singularities.
\end{theorem}

 Theorem \ref{thm:main} occupies a good part of
  this article and its proof culminates in Section \ref{sec:proof}.
 Proposition \ref{lem:unwell} gives conditions under which 
 hidden singularities are in fact well hidden yielding the following
 result.

\begin{theorem}
 \label{cor:canthide}
   If $r$ is a minimal symmetric descriptor realization with $r(0)\succ 0$, then
\begin{itemize}
  \item[(lim)] if $\fP_r(n)$ is convex for each $n$, then $\fP_r(n)$ contains no hidden singularities; and
  \item[(alg)] the algebraic domain of $r$ contains no hidden singularities.
\end{itemize}
\end{theorem}

\begin{remark}\rm   
 \label{rem:KvVpf}
   Item (alg) implicitly  appears
   in \cite{KvV}, though with a different proof than found here. 
  After developing background on free rational functions in 
  Section \ref{sec:rat}, we indicate how to obtain
  item (alg) from the results of \cite{KvV}. 

   A very special case of Theorem \ref{cor:canthide} (alg) also appears in \cite{HMV}. 
\end{remark}

\subsection{Reader's guide}
 \label{sec:guide}
  The remainder of the paper is organized as follows.
 Background on noncommutative polynomials and
 minimal symmetric descriptor realizations
 appears in Section \ref{sec:words}. Section \ref{sec:pre}
 reminds the reader of the computation of the  inverse
 of a block matrix
 via the Schur complement and describes a Fock
 space construction.  The proof of Theorem \ref{thm:main}
   begins in Section \ref{sec:hidden} with
  the  details of well hidden
  singularities and it concludes in Section \ref{sec:proof}. 
  The background on free rational expressions and their
  algebraic domains is the subject of Section 
  \ref{sec:rat} and can be skipped by the reader
 interested only in the limit domain. Theorem \ref{cor:canthide}
  is proved at the outset of Section \ref{sec:repCor}.
  Section \ref{sec:repCor} concludes with the 
  proof of Theorem \ref{cor:main}. 

\subsection{Acknowledgments}
 \label{sec:ack}
 We  thank Dima Kaliuzhnyi-Verbovetskyi 
for conversations tying this work to \cite{KvV}.
We also thank Jaka Cimpric for providing an argument
used in \S \ref{sec:KvVpf} and  Ben Greenberg for help with
 the preperation of this manuscript.

\section{A Few Words about Words and Minimal Realizations}
 \label{sec:words}
   Let $\cW$  denote the free semigroup on the $g$ freely 
 noncommuting formal variables \index{$\cW$} 
  $\{\fgg_1,\ldots,\fgg_g\}$.
An element $w$ of $\cW$ is a  \df{word} and takes the form
\[
   w = \fgg_{i_1} \ldots \fgg_{i_k}.
\]
  The \df{empty word}, $\emptyset$, \index{$\emptyset$}
 plays the role of the semigroup identity.
 There is a natural \df{involution ${}^T$} on $\cW$ which reverses  the
 order of a word, \index{$w^T$}
\[
 w^T = x_{i_k} \ldots x_{i_1}.
\]
 In particular, each variable itself is \df{symmetric}
 in the sense that $x_j^T=x_j$.
  A word $w$ is \df{evaluated} at a  tuple 
  $M=(M_1,\dots,M_g)$ of $n\times n$ matrices  by simply replacing
  $x_j$ by $M_j$ so that
\[
   w(M)=M^w = M_{i_1}\ldots M_{i_k}.
\]
 In the case of a tuple $X=(X_1,\dots,X_g)$ of symmetric matrices,
 the transpose operation on words is compatible with the
usual transpose operation on matrices in that $w(X)^T=w^T(X)$. 

  The book \cite{BR84} is an excellent reference 
 for additional details of the discussion of rational expressions found here. 
 The realization of the rational expression $r$ of Equation \eqref{eq:dreal} is \df{minimal} if 
\[
 \{(JA)^w JC\tth : w\in \cW, \ \ \tth \in \mathbb R^\ell\}
\]
 spans all of $\mathbb R^d$. 
 Equivalently, $r$ is minimal if 
 $u\in\mathbb R^d$ and 
\[
  ((JA)^w)^* JC^* u  =0
\]
 for all words $w$  implies $u =0$. 
 (See Lemma 4.1 in \cite{HMV}.)

 Another descriptor realization 
\[
 r_*  = D_* + C_*^T (J-L_{A_*}(x))^{-1} C_*
\] 
  is \df{equivalent} to $r$ if for each $n$ and each $X$ in both
 $\mfI_{J_*-L_{A_*}(x)}$ and $\mfI_{J-L_{A}(x)}$
 we have $r_*(X)=r(X).$ 
 Every descriptor realization is equivalent to a minimal descriptor
 realization. Hence, we may, and henceforth do,  assume without loss of generality
  that the realization $r$ is minimal. In fact, if
  $r_*$ is another minimal descriptor realization equivalent to $r$,
  then, as shown in Lemma \ref{lem:limits},    $\dome(r_*,n)=\dome(r,n).$ 
   From the definition of $\domalg(r,n)$ (which appears
  in Section \ref{sec:rat}) it is evident that $\domalg(r_*,n)=\domalg(r,n)$.
  While it is not needed in the sequel, we note that  any two minimal
   descriptor realizations with the same 
    $D$ term 
   are similar (in a precise sense we do not define here). (See \cite{BMG05} or \cite{HMV} for  examples.)

\section{Preliminaries: Schur complements and a Fock space construction}
\label{sec:pre}
 As preliminaries, in this section we remind the reader of the computation of
 the inverse of a matrix using Schur complements and
 introduce a Fock space construction.

\subsection{Schur complements}

\begin{lemma}
 \label{lem:inverse0}
  Suppose that the square matrix $M$ has the block form,
\[
    M= \begin{pmatrix} \Phi  & \Omega^T \\ \Omega & \Psi \end{pmatrix},
 \]
  $\Psi$ is invertible and let
\[
 \cS =\Phi - \Omega^T \Psi^{-1}\Omega.
\]
  The matrix $M$ is invertible
  if and only if $\cS$ is invertible and moreover
  in this case
\begin{equation}
 \label{eq:Slemma}
 \begin{split}
  M^{-1} 
   = & \begin{pmatrix} I & 0 \\ -\Psi^{-1} \Omega & I\end{pmatrix}
 \begin{pmatrix} \cS^{-1} & 0 \\0 & \Psi^{-1} \end{pmatrix} 
  \begin{pmatrix} I & -\Omega^T \Psi^{-1} \\ 0 & I \end{pmatrix} \\
 = & \begin{pmatrix} \cS^{-1} &     -\cS^{-1}\Omega^T \Psi^{-1} \\
       -  \Psi^{-1}\Omega \cS^{-1} & \Psi^{-1} + \Psi^{-1}\Omega \cS^{-1} \Omega^T \Psi^{-1} 
  \end{pmatrix}.
 \end{split}
\end{equation}

  In particular, if $\Psi$ is invertible,
  but  $\cS$ is not invertible, then $M$ is not invertible
 and moreover, for any vector $\zeta$ in the kernel of $\cS$,
 the vector
\[
 \begin{pmatrix} \zeta \\ - \Psi^{-1}\Omega \zeta \end{pmatrix}
\]
 is in the kernel of $M$.
\end{lemma}

\begin{remark}\rm
 \label{rem:sure}
 The matrix $\cS$ is the \df{Schur complement} of $M$ with 
  respect to $\Phi$. Alternately it is the Schur complement of $M$
 \df{pivoting} on $\Psi$. 

  Similarly, the Schur complement of $M$ with respect to $\Psi$ is
  \[
 \cS_* =\Psi - \Omega \Phi^{-1}\Omega^T
\]
  and in this case the analog of Equation \eqref{eq:Slemma} is
\begin{equation}
 \label{eq:Slemma-alt}
 \begin{split}
M^{-1} 
   = &  \begin{pmatrix} I & - \Phi^{-1} \Omega^T \\ 0 & I \end{pmatrix} 
 \begin{pmatrix} \Phi^{-1}  & 0 \\0 & \cS_*^{-1} \end{pmatrix} 
 \begin{pmatrix} I & 0 \\ -\Omega \Phi^{-1} & I\end{pmatrix} \\
 = & \begin{pmatrix} \Phi^{-1}+\Phi^{-1}\Omega^T \cS_*^{-1}\Omega\Phi^{-1} 
      & -\Phi^{-1}\Omega^T \cS_*^{-1} \\ - \cS_*^{-1} \Omega \Phi^{-1} & \cS_*^{-1}
  \end{pmatrix}.
\end{split}
\end{equation}
\end{remark}

\begin{proof}
  Direct calculation shows that the matrix $M$ can be written as
\[
 \begin{pmatrix} I & \Omega^T \Psi^{-1} \\ 0 & I \end{pmatrix}
 \begin{pmatrix} \cS & 0 \\0 & \Psi \end{pmatrix} 
  \begin{pmatrix} I & 0 \\ \Psi^{-1} \Omega & I\end{pmatrix}.
\]
 The second statement is also proved by direct calculation. 
\end{proof}

\subsection{A Fock space construction}
 \label{subsec:Fock}
 For a given positive integer $\nu$, let $\mathcal W_\nu$ denote the
 words of \df{length} at most $\nu$. (The empty word has length $0$.)
 Let $\mathcal F(\nu)$ denote the \index{$\mathcal F(\nu)$}
 Hilbert space with orthonormal basis $\mathcal W_\nu$. Thus
 $\mathcal F(\nu)$ is a truncated version of the standard
 Fock space \index{Fock space} on $g$ the freely noncommuting indeterminates $\{x_1,\dots,x_g\}$. 
 The dimension of $\mathcal F(\nu)$ is 
 $\mathfrak{d} = \sum_0^\nu g^j.$

\def\tw{{\tilde w}}

 Let $S=(S_1,\dots,S_g)$ \index{$S$} denote the shifts on $\mathcal F(\nu)$. Thus
 $S_j$ is determined by its actions on words $w\in\mathcal W_\nu$ by
\[
  S_j w = \begin{cases} x_j w & \mbox{ if } |w|<\nu \\
                          0 & \mbox{ if } |w|=\nu, \end{cases}
\]
 where $|w|$ \index{$\vert w \vert$} 
 denotes the length of the word $w$. 
 It is straightforward to verify that the adjoint of $S_j$ is
 determined by
\[
  S_j^* w = \begin{cases} \tw & \mbox{ if } w=x_j \tw \\
            0 & \mbox{ otherwise.} \end{cases}
\]
  Note that empty word, $\emptyset$, is a cyclic vector
 for the tuple $S$.  Let $K=K(\nu)$ denote the tuple\index{$K$}\index{$K(\nu)$}
 with $j$-th entry $K_j=S_j^*+S_j$. Thus each $K_j$
 is self adjoint and $K\in\mathbb S_{\mathfrak{d}}(\mathbb R^g)$.

\begin{lemma}
 \label{lem:fock}
 Fix  a word $\omega$ of length $\nu$. If $w\in\mathcal W_\nu$,
 but $w\ne \omega$, then $K^w \emptyset$ is orthogonal
  to $\omega$.  In particular, given a nonzero vector $\zeta$ in a  Hilbert
 space $\mathcal H$, there exists a mapping 
  $Q:\mathcal H\to \mathcal F(\nu)$ 
  such that $Q^T K^w\emptyset =0$ if $w\ne \omega$
 and $Q^T K^\omega \emptyset = \zeta$. 
\end{lemma}


\begin{proof}
  Fix a word $w\ne \omega$ of length $\mu\le \nu$ and write
\[
  w= x_{j_1} x_{j_2} \cdots x_{j_\mu}.
\]
 Because, when expanding $K^w$ as a sum of products
 of the $S_j$ and $S_k^*,$ every term, except for 
 $S^w$,  contains at least one adjoint term,
\[
  K^w\emptyset = S^w \emptyset + m = w+m,
\]
 where $m$ is a sum of words of length at most $\nu-2.$
 It follows that $\omega$ is orthogonal to $K^w\emptyset$. 
 Letting $\cG$ denote the span, in $\cF(\nu)$,  of the set
  $\{K^w\emptyset: w \in \cW_\nu, \ \  w\ne \omega\}$ it
  follows that $\omega$ is orthogonal to $\cG$. 
  In particular, letting $[\omega]$ denote the span
  of $\{\omega\}$ and $[\omega]^\perp$ its orthogonal complement,
  $\mathcal G\subset [\omega]^\perp$.  Define $Y:\cF(\nu)\to H$ 
  by $Y\omega = \zeta$ and $Y=0$ on $[\omega]^\perp$.   
  Then $YK^w\emptyset =0$ for $w\ne \omega$.
  Choosing $Q=Y^T$ completes the proof.
\end{proof}

\section{Well Hidden Singularities and  Perturbations}
 \label{sec:hidden}
   The proof of Theorem \ref{thm:main} begins  here 
  and concludes in Section \ref{sec:proof}.
  It   proceeds by
  contradiction. Accordingly, 
 suppose $r$, the given minimal symmetric descriptor realization
 as in Equation \eqref{eq:dreal}, has a well hidden singularity  $\hX$
  which is now fixed through the end of Section \ref{sec:proof}.
 \index{$\hX$} In particular, $J-L_A(\hX)$ is singular
  and thus has a nontrivial kernel $\mathcal K.$ \index{$\mathcal K$}
  Let $N$ denote the size of $\hX$ so that $\hX \in \mathbb S_N(\mathbb R^g)$.

 Recall that $J-L_A(x)$ is a $d\times d$ affine linear pencil. 
 With respect to the decomposition of $\mathbb R^d \otimes \mathbb R^N$
 as $\mathcal K\oplus \mathcal K^\perp$, write
\[
  J-L_A(\hX) = \begin{pmatrix} 0  & 0 \\ 0 & R \end{pmatrix}
\]
 and
\[
  - L_A(\hX) =\begin{pmatrix} \alpha & \beta^T \\ \beta & R_1
  \end{pmatrix}.  \index{$V$}
\]
   Letting $\bV$ denote the inclusion of $\mathcal K$ into 
 $\mathbb R^d \otimes\mathbb R^N$, the matrix $\alpha$ is given by
 $-\bV^T L_A(\hX)\bV$.  Note that $R$ is invertible.

\begin{lemma}
 \label{lem:Xdirection}
  The pencil $J-L_A(\hX+t^2\hX)$ is invertible for $t$ sufficiently 
  close to $0$, but $t\ne 0$.
\end{lemma}

\begin{proof}
  If not, then $\det(J-L_A(\hX+t^2\hX)),$ being a polynomial, is identically $0$.
  Hence $\det(J-L_A(\hX+u\hX))$ is identically $0$ and thus $0$ at $u=-1,$
  which gives the contradiction $\det(J)=0$. 
\end{proof}

 From the definitions,  
\[
 J-L_A(\hX+t^2\hX) = \begin{pmatrix} t^2\alpha & t^2\beta^T\\
   t^2\beta & R(t) \end{pmatrix},
\] 
  where $R(t)=R +  t^2 R_1$. 
 Since, for $t\ne 0$ but near $0$, both 
  $J-L_A(\hX+t^2\hX)$  and $R(t)$ are invertible, Lemma \ref{lem:inverse0} implies that 
\begin{equation}
 \label{eq:F}
  F(t) := [\alpha - t^2\beta^T R(t)^{-1} \beta]
\end{equation}
 is invertible.  In particular, $\det(F(t))$ is not identically equal to zero.
  Thus Cramer's rule implies that there is a non-negative integer $p$ 
  and a nonzero matrix $\bM$ such that
\begin{equation}
 \label{eq:defbM}
  \lim_{t\to 0} t^p F(t)^{-1} = \bM; \index{$\bM$}
\end{equation}
  i.e., the limit exists and is nonzero. In fact $p$ 
  \index{$p$} is an even
  integer, since
  $F$ is really a function of $t^2$. For future use,
  let $q=\frac{p+2}{2}$ so that $p=2q-2$.  \index{$q$}
 If $\alpha$ is invertible, then $p=0$ and $q=1$
  and the  arguments to come are much simpler. 
 
\subsection{A further perturbation}
  Let $M$ be a positive integer and fix $K\in\mathbb S_M(\mathbb R^g).$
  Let $\smatMNg$ denote the set of 
  $g$-tuples $H=(H_1,\dots, H_g)$ of $M\times N$ matrices.
  Given such an $H$ and real numbers $s,t,\rho$,  let
\begin{equation}
  \label{eq:deftildeX}
 \thX(s,t,\rho,H)= \thX = \begin{pmatrix} \hX+t^2\hX & st^q H^T \\ 
    st^q H & \rho K 
 \end{pmatrix} \index{$\thX$}
\end{equation}
  and observe, since $\hX$ is a well hidden singularity,
  $\thX$ is a hidden singularity for $t=0$ and $\rho$ sufficiently near $0.$ 
  
 Write, with respect to the decomposition of 
  $(\mathbb R^d \otimes \mathbb R^N )\oplus (\mathbb R^d \otimes \mathbb R^M)$
 as $\mathcal K\oplus \mathcal K^\perp\oplus [\mathbb R^d \otimes \mathbb R^M]$,
\begin{equation}
 \label{eq:tryq}
  J-L_A(\thX) = \begin{pmatrix} t^2\alpha & t^2\beta^T & st^q\gamma^T\\
    t^2\beta & R(t) & st^q W^T \\ st^q \gamma & st^q W & 
     J-  \rr L_A(K) 
  \end{pmatrix}.
\end{equation}
   Here $\gamma$ and $W$ come from decomposing $L_A(H^T)$ and 
  so are linear maps (matrices). 

\begin{lemma}
 \label{lem:tXinDr}
  Given $H$ and $K$ (as above), there is a $\rho_0$ such that
  for each $|\rho|<\rho_0$ there is an $s_0$ such that for
  each $|s|<s_0$ there is a $t_0$ such that for $0<|t|<t_0$,
  the matrix $J-L_A(\thX)$ is invertible
  (thus $\thX=\thX(\rho,s,t,H,K)\in \mfI_{J-L_A(x)}$ for such $\rho,s,t$).
\end{lemma}

\begin{proof}
      Since $\hX$ is a well hidden singularity, there
  exists a $\rho_0$ such that if $|\rho|<\rho_0$, then
   $\hX\oplus \rho K$ is a hidden singularity. Moreover,
  $\rho_0$ can be chosen small enough that 
  $\rho K$ is in the component of zero of the
  invertibility set of $J-L_A(x)$.

  By Lemma
  \ref{lem:Xdirection} there is a $u$ so that
   $J-L_A(\hX+u^2 \hX)$ is invertible. 
   It follows that with $\rho$ given and 
   this $u$ there is an $s_0>0$ such that if $|s|<s_0$
  then, $J-L_A(\thX(s,\rho,u,H))$ is invertible. 
  In particular,
  with $\rho$ and such an $s$ fixed, as a function of
   $t$ the matrix-valued polynomial $\tau(t)=J-L_A(\thX(s,\rho,t,H))$
  is invertible at $t=\pm |u|$ and hence  fails to 
   be invertible at most finitely many times. Consequently,
  there is a $t_0>0$ so that for $0<|t|<t_0$ it is invertible
  (with of course $\rho$ and then $s$ fixed).  
 \end{proof}

  Our aim is to use the fact that for, $t=0$ and $\rho$ small,
  $\thX$ is a hidden singularity and  examine
  properties of $r(\thX)$ as $t$ tends to $0$
  (with the other variables fixed). Accordingly, 
  let $\Gamma$ \index{$\Gamma$} denote the lower right two by two block of
  the matrix in Equation \eqref{eq:tryq}. Thus
 \[
  \Gamma = \begin{pmatrix} R(t) & st^q W^T \\
    st^q W & Y\end{pmatrix},
\]
 where \index{$Y$}
\begin{equation}
 \label{eq:Y}
 Y=J- \rr L_A(K). 
\end{equation}
 Let $\Delta$ \index{$\Delta$} denote the the Schur complement of $\Gamma$ (relative
  to  $R(t)$),
\begin{equation}
 \label{eq:Delta}
 \Delta = Y- s^2 t^{2q} W R^{-1}W^T.
\end{equation}
 Since $Y$ is invertible, $\Delta$ is invertible for $t$ sufficiently close to $0$. 

For notational ease, write 
\[
  (J-L_A(\thX)) = \bem
 t^2 \alpha & t^2 \zeta^T \\
  t^2 \zeta & \Gamma
  \eem
  \qquad  \mbox{with} \qquad
  \zeta =   \begin{pmatrix} \beta \\ st^{q-2} \gamma \end{pmatrix}
\]
  and, recalling the definition of $F$ from Equation \eqref{eq:F}, let
\begin{equation}
 \label{eq:F*}
 F_* = F - s^2 t^{2q-2} (\gamma^T - t^2 \beta^T R^{-1}W^T) \Delta^{-1}
   (\gamma -t^2 W R^{-1}\beta).
\end{equation}
 To verify that $F_*$ is the Schur complement for  
$ (J-L_A(\thX))$ pivoting on $\Gamma$ (see Equation \eqref{eq:Slemma-alt}),
  first observe that 
  from Equation \eqref{eq:Slemma},
\begin{equation}
 \label{eq:Gammaminus1}
  \Gamma^{-1} = 
  \begin{pmatrix} I & -st^q R(t)^{-1}W^T \\ 0 & I\end{pmatrix}
  \begin{pmatrix} R(t)^{-1} & 0\\ 0 & \Delta^{-1} \end{pmatrix}
  \begin{pmatrix} I & 0 \\ - st^q WR(t)^{-1} & I \end{pmatrix}.
\end{equation}
   Thus
\begin{equation}
 \label{eq:xF*}
 \begin{split}
     \zeta^T
      \Gamma^{-1}
       \zeta
   = \beta^T R^{-1} \beta \ + \  & (st^{q-2} \gamma^T -st^q \beta^T R^{-1}W^T) \Delta^{-1}
   (st^{q-2}\gamma -  st^q W R^{-1}\beta).
 \end{split}
\end{equation}
  Hence,
\begin{equation}
 \label{eq:F**}
    \alpha - t^2 \zeta \Gamma^{-1}\zeta^T 
     = F_*. 
\end{equation}
 Using \eqref{eq:F**},  an application of Lemma \ref{lem:inverse0}
    gives,
\begin{equation}
 \label{eq:Jminus}
  (J-L_A(\thX))^{-1}
   = \begin{pmatrix} (t^2 F_*)^{-1}  & - (t^2F_*)^{-1} t^2 \zeta^T \Gamma^{-1} \\
     -  \Gamma^{-1} t^2 \zeta (t^2 F_*)^{-1} & 
     \Gamma^{-1} +  t^2  \Gamma^{-1} \zeta  F_*^{-1} \zeta^T \Gamma^{-1}
   \end{pmatrix}.
\end{equation}

Our immediate goal is to analyze $   \hk^T r(\thX) \hh $  where
 \beq
 \label{eq:hats}
 \hh = 0 \oplus \tth \otimes h, \ \
 \hk = 0 \oplus  (\ttk\otimes k)  \in [\mathbb R^\ell \otimes \mathbb R^N] \oplus  (\mathbb R^\ell\otimes \mathbb R^M)
    = \mathbb R^{\ell N}\oplus (\mathbb R^\ell \otimes \mathbb R^M)
 \eeq 
  for given vectors 
 $h,k\in\mathbb R^M$ and  $\tth,\ttk \in\mathbb R^\ell.$ This
  notation we will carry throughout.
 From  Equation \eqref{eq:Jminus},
\beq
\label{eq:hrk}
 \begin{split}
 \hk^T [r(\thX)   - D\otimes I] \hh   
   = & \bem
0\\ C\ttk \otimes k
\eem^T
   (J-L_A(\thX))^{-1}
  \bem
0\\ C\tth \otimes h
\eem \\   
= & (C\ttk \otimes k)^T
  [\Gamma^{-1} + t^2  \Gamma^{-1} \zeta  F_*^{-1} \zeta^T \Gamma^{-1} ]
   C\tth \otimes h.
\end{split}
\eeq

\subsection{Taking a limit}
  To analyzing the limit
 $\lim_{t\to 0}r(\thX),$ we first consider the limit
 $\lim_{t\to 0} t^p F_*^{-1}(t)$. 
   For notational convenience, write
\[
 F_* = F - s^2 t^{2q-2} 
 G(s,t)
\]
 where
\[
  G(s,t) = (\gamma^T - t^2 \beta^T R^{-1}W^T) \Delta^{-1}
   (\gamma -  t^2 W R^{-1}\beta).
\]
   Observe, from Equation \eqref{eq:Delta}, that for $s\ne 0$ fixed, 
\[
  G_0 = \lim_{t \to 0} G(s,t) =\gamma^T Y^{-1} \gamma,
\]
which we note, in view of Equation \eqref{eq:Y}, is independent of of $s.$
 Moreover, with these notations,
\[
  F_* = F(t) - s^2 t^p G(s,t)
  = F(t) [I - s^2 t^p F(t)^{-1}G(s,t)].
\]
It follows that, for $s\ne 0$ sufficiently close to $0$, 
\[
 \begin{split}
\label{eq:defH}
 \eta(s):=& \lim_{t\to 0} t^p F_*^{-1} \\
   = & \lim_{t\to 0} (I- (t^pF^{-1})s^2 G)^{-1} t^p F^{-1}\\
  = & \lim_{t\to 0} (I-s^2 \bM G_0(s))^{-1} \bM.
\end{split}
\]
 Further,
\begin{equation}
 \label{eq:limH}
 \lim_{s\to 0} \eta(s):= \bM.
\end{equation}


\begin{lemma}
 \label{lem:ifind0}
    For $s$ fixed (and sufficiently small), 
\beq
\label{eq:gamlim}
 \lim_{t\to 0}
   \bem
  0 & I_{d M}
  \eem
  \Gamma^{-1} 
   \bem
  0 \\ I_{d M}
  \eem
 =   
 (J -  \rr L_A(K))^{-1} = Y^{-1}
\eeq
 and 
\beq
 \begin{split}
 \label{eq:ifind0}
  \lim_{t\to 0} t^2 
  \bem
  0 & I_{d M}
  \eem
   \Gamma^{-1} \begin{pmatrix}\beta \\ st^{q-2} \gamma
  \end{pmatrix}& F_*^{-1} \begin{pmatrix} \beta^T & st^{q-2}
  \end{pmatrix}  \Gamma^{-1}
   \bem
  0 \\ I_{d M}
  \eem
  \\
  = &
  s^2
   Y^{-1} \gamma \; \eta(s) \;  \gamma^T Y^{-1}
  \\
  = &
 s^2 
  Y^{-1} \gamma 
( I -  s^{2} \bM   \gamma Y^{-1} \gamma^T)^{-1} \bM 
  \gamma^T Y^{-1}.
 \end{split}
\eeq
\end{lemma}

\begin{remark}\rm
 \label{rem:Q0onlyX}
   Importantly, $\bM $ depends only upon $\hX$ (and the given realization for $r$)
   and not upon $H,K$.  On the other hand, $Y$ depends on $K$ 
  and $\gamma$ depends upon $H$. 
  Later we shall see that \eqref{eq:ifind0} is 0.
\end{remark}

\begin{proof}
  Equation \eqref{eq:gamlim} follows from Equation \eqref{eq:Gammaminus1}
  and the definition of $\Delta$ given in Equation \eqref{eq:Delta}.

  Moving onto Equation \eqref{eq:ifind0}, note that 
\[  t
  \bem  0 & I_{d M}  \eem
    \Gamma^{-1} \begin{pmatrix} \beta \\ st^{q-1} \gamma \end{pmatrix}
   =  st^{q-1} 
    \Delta^{-1} (\gamma -t^2 WR^{-1}\beta).
\]
 Thus, we need to compute
\[
 \lim_{t\to 0} s^2 t^{2q-2}
 \Delta^{-1} (\gamma -t^2 WR^{-1}\beta)
   F_*^{-1}  (\gamma^T -t^2 \beta^T R^{-1}W^T )\Delta^{-1}.
\]
 Now  $t^{2q-2}F_*^{-1}=t^p F_*^{-1}$ converges to $\eta(s)$ as defined
  in Equation \eqref{eq:defH} and
\[
 \lim_{t\to 0} 
  \Delta^{-1} (\gamma -t^2 WR^{-1}\beta)
   = 
   \lim_{t \to 0}\Delta^{-1}\gamma = Y^{-1}\gamma.
\]
 Thus the relevant limit exists and
 is
\[
   s^2 
    Y^{-1} \gamma \eta(s) \gamma^T Y^{-1} 
\]
 as claimed. 
\end{proof}

\subsection{A limit formula}
  The proof of the following proposition is based upon the observation that,
  with $\hh$ and $\hk$ defined in \eqref{eq:hats}, 
\[
  \lim_{t\to 0} \hk^T r(\thX) \hh = (\ttk\otimes k)^T r(\rho K) \tth\otimes h
\]
 is independent of $H$ (and $s$) for $\rho$ sufficiently close to zero.  
 This is because for $\rho$ and $s$ fixed appropriately, 
\beq
\label{eq:limnos}
 \lim_{t\to 0} r(\thX) = r(\hX\oplus \rho K) = r(\hX)\oplus r(\rho K).
\eeq
Indeed this is the key use of the hypothesis that 
the singularity $\mathbb X$ is well  hidden.

\begin{proposition}
\label{prop:indsrho}
  Given $K\in\smatMg$ and $H\in\smatMNg$,
 there exists a $\rho_0>0$ such that for each $|\rho|<\rho_0$
  there is an $s_0$ such that for $|s|<s_0$  
\beq
\label{eq:core}
 s^2 (C\ttk \otimes  k) Y^{-1} \gamma 
( I -  s^{2} \bM   \gamma Y^{-1} \gamma^T)^{-1} \bM 
  \gamma^T Y^{-1} 
 (C\tth \otimes h)=0.
\eeq
 In particular, 
\[
  (C\ttk \otimes  k) Y^{-1} \gamma  \bM 
  \gamma^T Y^{-1} 
 (C\tth \otimes h)=0.
\]
\end{proposition}

\begin{proof}
 By  \eqref{eq:limnos}  the left hand side of  \eqref{eq:hrk}
\[    
  \hk^T [r(\thX)-D\otimes I] \hh =
  (C\ttk \otimes \hk)^T
  \Gamma^{-1}  (C\tth \otimes \hh)
 \  + \
   t^2   (C\ttk \otimes \hk)^T \;
   ( \Gamma^{-1} \zeta  F_*^{-1} \zeta^T \Gamma^{-1} )
  \;  (C\tth \otimes \hh),     
\]
has 
 a limit at $t=0$ independent
  of $s$ and $H$.  Thus the right side does too.
On the right hand side the limit of the first (left most) term is handled by 
Equation \eqref{eq:gamlim} and is, by inspection,
 independent of $s$ and $H$.
 Hence the limit of the  second term
 on the right hand side  is too.
  From 
Equation \eqref{eq:ifind0} of Lemma \ref{lem:ifind0},
the limit of this second term is the left side of Equation
\eqref{eq:core} which is thus independent of $s$. 
 Hence the left hand side of Equation \eqref{eq:core}
 is constantly equal to its value at $0$, namely $0$.
   Hence \eqref{eq:core} holds.
\end{proof}

\section{Proof of Theorem \ref{thm:main}} 
\label{sec:proof}
 Recall we have assumed that the minimal symmetric descriptor
 realization $r$ has the well hidden singularity $\hX.$
 Using the perturbation $\hX+t\hX$ we constructed
 a nonzero matrix $\bM$ defined by Equation \eqref{eq:defbM}.
 In this section we reach
 the contradiction $\bM=0,$ and deduce
  that $\hX$ was not in fact a well hidden singularity thus
 completing the proof of Theorem \ref{thm:main}.

  Recall that $Y= J-\rho L_A(K)$ depends upon both
  $\rho$ and $K$ and $\gamma$ depends on $H$.
  (On the other hand, $\bM$ depends only on $\hX$.)
  From Proposition \ref{prop:indsrho} 
\beq
\label{eq:prepresummary}
 (C\ttk \otimes k)^T  Y^{-1}\gamma \bM  \gamma^T Y^{-1} (C\tth \otimes h) = 0,
\eeq
 for all $K\in\smatMg,$ all $H\in\smatMNg$ and all  $\rho$ sufficiently small.
 Here, as in Proposition \ref{prop:indsrho}, $h,k\in\mathbb R^M$ 
  and $\tth,\ttk\in\mathbb R^\ell$. 

  Given a linear mapping $Q:\mathbb R^N \to\mathbb R^M$, choose
   $H=Q\hX=(Q\hX_1,\cdots,Q\hX_g)$ in 
  Equation \eqref{eq:prepresummary}. Recalling
  the definition of $\bV$ given at the outset of Section 
  \ref{sec:hidden}, note that 
   $L_A(Q\hX)=    (I_d \otimes Q) L_A(\hX)$ and further
\[
  \gamma = (  I_d\otimes Q) L_A(\hX)\bV =  (  I_d\otimes Q) (J \otimes I_N) \bV
\]
 because $L_A(\hX)= (J \otimes I_N) $ on $\mathcal K.$
Henceforth we abbreviate and  use $Q$ to denote $I_d\otimes Q$ 
and $J$ to denote  $J \otimes I_N$
(in accordance with usual practice). 
  Thus, by \eqref{eq:prepresummary},
\[
(C\ttk \otimes k)^T 
         ([J-\rr L_{A}(K)]^{-1} Q J \bV  \bM \bV^T J Q^T [J-\rr L_A(K)]^{-1} )
      C\tth \otimes h = 0.
\]
  It follows that
\begin{equation}
 \label{eq:main-alt}
(C\ttk \otimes k)^T 
         (J[I-\rr L_{AJ}(K)]^{-1}  Q J \bV  \bM  \bV^T J Q^T [I-\rr L_{JA}(K)]^{-1} J )
      C\tth \otimes h = 0.
\end{equation} 
 Write Equation \eqref{eq:main-alt} as  a power series in $\rho$ 
and use that   every coefficient must be zero to obtain, for each 
  non-negative integer $\nu$,
\begin{equation}
 \label{eq:firstoftwo}
  \sum_{j=0}^\nu (JC\ttk \otimes k)^T L_{AJ}(K)^j  Q J \bV\bM \bV^T J Q^T L_{JA}(K)^{n-j} (JC\tth \otimes h) = 0.
\end{equation}

  Fix words $\omega_1$ and $\omega_2$
  words of length $\nu_1$ and $\nu_2$ respectively 
  and let $\nu=\nu_1+\nu_2$.  
  In the notation of Subsection \ref{subsec:Fock}, let 
\[
 \cF = \cF(\nu_1) \oplus \cF(\nu_2)
\]
 and let  $K$ denote the tuple of self-adjoint matrices acting 
 on $\cF$ defined by
\[
  K_j =\begin{pmatrix} K_j(\nu_1) & 0 \\ 0 & K_j(\nu_2) \end{pmatrix}.
\]
 Thus, $K$ can viewed as an element of $\mathbb S_M(\mathbb R^g)$
 where $M$ is the dimension of $\cF$ (and can be computed
  explicitly in terms of $\nu_1,\nu_2$ and $g$).  
 Recall that $\hX$ is acting on $\mathbb R^N$ 
  and fix vectors $\zeta_1, \zeta_2\in\mathbb R^N$. 
  Using Lemma \ref{lem:fock},  define 
  $Q_j:\mathbb R^N\mapsto \cF(\nu_j)$ 
  such that $Q_j^T w =0$ if $w\ne \omega_j$
 and $Q_j^T \omega_j = \zeta_j$.  Define $Q:\mathbb R^N \to \cF$ by
\[
  Q^T = \begin{pmatrix} Q_1^T &  Q_2^T \end{pmatrix}:\cF \to \mathbb R^N.
\]
  Finally let $h= \emptyset \oplus 0\in\cF$ and 
 $k= 0\oplus \emptyset \in\cF.$

  For the choices in the previous paragraph 
   consider, for $0\le j \le \nu$,
\[
 \begin{split}
 Q^T  L_{JA}(K)^{\nu-j} (JC\tth \otimes h)
  = & \, Q^T \sum_{|w|=\nu-j} (JA)^w JC\tth  \otimes K^w h  \\
  = & \, Q^T \sum_{|w|=\nu-j} (JA)^w JC \tth \otimes (K^w(\nu_1)\emptyset\oplus 0) \\
  = & \sum_{|w|=\nu-j} (JA)^w JC \tth \otimes Q_1^T K^w(\nu_1)\emptyset.
\end{split}
\] 
 Hence, 
\[
  Q^T  L_{JA}(K)^{\nu-j} (JC\tth \otimes h) = \begin{cases} 0 & \mbox{ if } \nu-j < \nu_1\\
                  (JA)^{\omega_1} JC \tth \otimes \zeta_1 & \mbox{ if } \nu-j=\nu_1 \\
                   * & \mbox{ if } \nu-j>\nu_1. \end{cases}
\]
 Likewise,
\[
  Q^T L_{JA}(K)^j (JC\ttk \otimes k)= 
  \begin{cases} 0 & \mbox{ if } j<\nu_2 = \nu-\nu_1 \ \ (\nu-j>\nu_1)\\
                (JA)^{\omega_2} JC\ttk  \otimes \zeta_2 & \mbox{ if } j=\nu_2 \ \ (\nu-j=\nu_1) \\
                 * &  \mbox{ if } j>\nu_2, \end{cases}
\]
  where $*$ is an expression which plays no role in the argument.
  In particular, 
\[
\begin{split}
(JC\ttk \otimes k)^T (L_{AJ}(K))^j &  Q J \bV\bM \bV^T J Q^T (L_{JA}(K))^{\nu-j} (JC\tth \otimes h) \\
  =& \begin{cases} [(JA)^{\omega_2} JC\ttk \otimes \zeta_2]^T J \bV\bM \bV^T J [(JA)^{\omega_1} JC \tth \otimes \zeta_1] & \mbox{ if } j=\nu_2 \\
            0 & \mbox{ if } j\ne \nu_2. \end{cases}
\end{split}
\]
 Hence, from Equation \eqref{eq:firstoftwo}, it follows that
\[
  [(JA)^{\omega_2} JC\ttk \otimes \zeta_2]^T J \bV\bM \bV^T J [(JA)^{\omega_1} JC \tth \otimes \zeta_1] =0
\]
  for all choices of words  $\omega_j$ vectors $\zeta_j\in\mathbb R^N$
  and $\ttk,\tth\in\mathbb R^\ell$. 
  The minimality assumption on
 the descriptor representation this implies that $J\bV\bM \bV^TJ=0$ which
  in turn leads to the contradiction $\bM=0$ and completes the proof.
  (Here we have used $J$ is invertible which gives $\bV\bM \bV^T=0$.
  Now $\bV$ is the inclusion of $\mathcal K$ into $\mathbb R^\ell\otimes \mathbb R^N$,
  so that $\bV^T$ is onto $\mathcal K$.)

\section{Free rational functions and their domains}
\label{sec:rat}
  With the proof of Theorem \ref{thm:main} complete,
 it remains to establish Theorems \ref{cor:canthide}
  and \ref{cor:main}.   Each of these theorems
 naturally splits into two statements, one about the 
 limit domain and the other about the algebraic domain
 of the minimal symmetric descriptor realization $r$.
 This section gives  the needed background on the algebraic
 domain of $r$.  Readers interested only in the limit
 domain can safely skip to Section \ref{sec:repCor}
 where the theorems are proved.

\subsection{Noncommutative Polynomials}
  The construction of free rational expressions begins with polynomials.
   Recall, from Section \ref{sec:words} that $\mathcal W$ 
   denotes the free semigroup on the $g$ letters $x=(x_1,\dots,x_g)$
  and that ${}^T$ is the involution on $\mathcal W$ which 
  reverses the order of a word.  

  A free polynomial is then an $\mathbb R$ linear combination 
  of words from $\mathcal W$ and we let $\Rx$ denote the
  collection of free polynomials.  Hence $p\in\Rx$ has the form
\[
  p = \sum p_w w,
\]
 where the sum is finite. 
  Evaluations on $\mathcal W$
  extend to $p\in\Rx$   in the obvious way as
\[
  p(X) = \sum p_w X^w.
\]
  for $X\in\smatng$.

  A free $k_1\times k_2$ matrix-valued polynomial $p$ can be viewed either as
  a $k_1\times k_2$ matrix with entries from $\Rx$ or as a (finite) linear
  combination of words with (real) $k_1\times k_2$ matrix coefficients,
\[
  p = \sum P_w w.
\]
 In the first case one evaluates $p(X)$ entrywise and in the second case
 one has 
\[
  p(X) = \sum P_w \otimes X^w.
\]
  Note that, for
  $Z_n=(0_n,0_n,\dots,0_n) \in \smatng$ where
  each $0_n$ is the $n\times n$ zero matrix,
  $p(Z_n)=I_n\otimes p(Z_1)$. In particular,
  $p(Z_n)$ is invertible for all $n$ or no $n$.
  Because of this simple
  relationship, in the sequel we will often simply write $p(0)$
  with the size $n$ unspecified.

  The involution on $\cW$ naturally extends to matrix-valued polynomials as
\[
   p^T = \sum P_w^T w^T.
\]
  The polynomial $p$ is \df{symmetric} if $p^T=p$.  
  Observe the involution is compatible with
  evaluation in that  $p^T(X)=p(X)^T$ and, if $p$ is symmetric,
 $p(X)^T=p(X)$; i.e., $p$ takes symmetric values.

\subsection{Rational Expressions and their Formal Domains}
  \label{item:intorat3}
  We use recursion to define the notion of a
  {\bf free (noncommutative) rational expression $r$
  analytic at $0$}
    and its value  $r(0)$ at $0$.
  This class includes free matrix-valued polynomials and $p(0)$ is the
  value of $p$ at $0.$ 
  If $p$ is $k\times k$ (square) matrix-valued and the
  $k\times k$ matrix $p(0)$ is invertible, then $p$ is invertible, its
  inverse is a free rational expression analytic at $0$,
  and $p^{-1}(0)=p(0)^{-1}$.
  Formal sums and products of free rational expressions
  analytic at $0$ with value at $0$ are defined accordingly,
  subject to the provision that the matrix sizes are compatible.
  Finally, a free rational expression $r$ analytic at $0$
  can be inverted  provided $r$ is  $k\times k$ (square)
  matrix-valued  and $r(0)$ is invertible. In this case, its inverse
  is a free rational expression, and $r^{-1}(0)=r(0)^{-1}$.

The \df{formal domain} in $\smatng$
  of a free rational expression $r$, denoted $\domfor(r,n)$, \index{domfor}
  is defined inductively. If $p$ is a polynomial, then it its
  formal domain  is all of $\smatng.$ If $r$ is the inverse
  of the polynomial $p$, then the formal domain of $r$ 
  is $\{ X \in \smatng : p(X) \text{ is an invertible matrix}\}$.
The formal domain of
a general free rational expression $r$
is equal to the intersection of the  formal domains $\domfor(r_j,n)$
for
the rational  expressions $r_j$ and, as necessary, 
their inverses  which
appear in the construction of the expression $r$.
 Note that by assumption $0\in\domfor(r,n)$.  
  Let $\domfor(r)$
 denote the sequence $(\domfor(r,n))_n$.

\subsection{ Equivalent Rational Expressions: Rational Functions }
\label{sec:ratFun}

  Note that two different expressions, such as
\begin{equation}
\label{eq:exr1r2}
       r_1= x_1 (1 - x_2 x_1)^{-1} \ \ \ \
\text{and} \ \ \ \
       r_2 = (1 - x_1 x_2)^{-1}x_1
\end{equation}
can be converted to each other using the rational operations
 described above.  Thus it is natural to 
  specify an equivalence relation on rational expressions.
 There are various ways of doing this
 and several
   are mentioned  in a paragraph  
in the introduction of [KvV] and the references
associated to it.  
 The notion used here mostly is that of  evaluation equivalence
  as found in \cite{HMV} and given in more detail in \cite{KvV}.
  While it is not needed for the results here,
  the notion of series equivalence is used 
  in Subsection \ref{sec:KvVpf} in connection with Remark \ref{rem:KvVpf}.
  Next we briefly describe these notions which turn out to
  be the same. 
  
  It is clear how to evaluate a free rational expression $r$
   on any $X \in \domfor(r)$.
   We can use these evaluations  to define an equivalence on free
   rational expressions which  call evaluation equivalence.
   Two free rational expressions $r$ and $s$ analytic at
   $0$ are {\bf evaluation equivalent} provided
   $r(X)=s(X)$ for each $n$ and each $X$ in the Zariski open
   set $\domfor(r,n) \cap\domfor(s,n)$.
   We reiterate that the $X_j$ are symmetric matrices.

\begin{remark}\rm
  The fact that both $r$ and $s$ are analytic at $0$ means that
  for each dimension $n$, the $0$ matrix $g$-tuple
  is in the intersection of their domains.
  Without this requirement that
  $r$ and $s$ are analytic at $0$
  it is possible that for certain
  $n$ one or both of the domains $\domfor(r,n)$  or $\domfor(s,n)$
  could be empty. Indeed, from the theory of polynomial
  identities, there are free polynomials which, for certain $n$,
   fail to be invertible on all of $\smatng$. 
\end{remark}

  The \df{algebraic domain} of rational expression $r$ (analytic at $0$) in $\smatng$, 
  denoted $\domalg(r,n),$ is the union of the formal domains
  for rational expressions $s$ (analytic at $0$) which are equivalent to $r$. 
   (We note that the equivalence class of a rational
  expression $r$ is what is typically called  a \df{free rational function} 
  $\mathfrak{r}$ and direct the reader to \cite{HMV} for further details.)
  \index{$\mathfrak{r}$}

\begin{lemma}
 \label{lem:limits}
   Let $r$ be a rational expression analytic at $0$.  If $\chi \in\domalg(r,n)$, 
  and $r_*$ is a rational expression equivalent to $r$ such that $\chi\in\domfor(r_*,n)$, then
\[
 \lim_{X\to \chi} r(X) = r_*(\chi),
\]
   where the limit is taken through $X\in\domfor(r,n)$.  In particular, if $\chi \in\domalg(r,n)$,
  then $\chi \in\dome(r,n)$. 
\end{lemma}

\begin{proof}
  Self evident.
\end{proof}

  Another notion of equivalence  comes from expanding two rational expressions
  $r$ and $s$  in a power
  series about $0$. The rational expressions $r$ and $s$ are \df{series equivalent} if
  the coefficients of their power series are the same. 
  Given $\epsilon >0$ and a positive integer $n$ let
\[
   N_\epsilon(n) = \{X\in\smatng : \sum X_j^2 \prec \epsilon I_n\}.
\]
  The \df{free $\epsilon$-neighborhood of $0$}
  \index{free neighborhood of $0$} \index{$N_\epsilon$} 
  (in $(\smatng)_n$ is the sequence of sets $(N_\epsilon(n))_n$. 
  In \cite{HMV} (see Proposition A.7) it is proved that series equivalence agrees with
  evaluation equivalence on some free $\epsilon$ neighborhood
  of $0$ (inside the principal component of
  $\domfor(r,n) \cap \domfor(s,n)$).

\def\hr{{\hat r}}

\subsection{The \cite{KvV} proof of Theorem \ref{cor:canthide} (alg) }
\label{sec:KvVpf}
We conclude this section by completing the tie between this paper
and \cite{KvV} where (alg) for rational expressions in free variables,
as opposed to the symmetric variables found here, is proved. 

 By simply ignoring the transpose operation on polynomials 
  and allowing evaluations on tuples of not necessarily self adjoint
 matrices one obtains the notion
 of free variables. 
For clarity, let $\{z_1,\dots,z_g\}$ denote freely noncommuting variables.
 Substituting $z_j$ for $x_j$ in Equation \eqref{eq:dreal} gives,
\begin{equation}
 \label{eq:drealz}
 r(z) = D+ C^T(J-L_A(z))^{-1}C.
\end{equation}
 Let $M_n(\mathbb R^g)$ denote the set of $g$-tuples $Z=(Z_1,\dots,Z_g)$
 of $n\times n$ matrices.  The descriptor realization $r$ is 
 naturally evaluated at $Z$ so long as $J\otimes I - L_A(Z)$ is
 invertible. In fact, this same observation holds for any
 rational expression giving rise to more expansive notions
 of domains (formal, algebraic, etc.) for rational expressions (analytic at $0$)
 as sequences whose $n$-th term is a subset of $M_n(\mathbb R^g)$. 

 Conversely, if $r$ is a rational expression in free variables which takes
  symmetric values when evaluated on tuples $X$ of symmetric matrices
  in the component of $0$ of its formal domain, then $r$ is 
  uniquely determined by its values on such tuples as we now
  explain.  

 
 What we must check to validate this claim is show
  that if $s$ is another
 rational expression in free variables, analytic at $0$,
  which agrees with $r$
 on symmetric matrices in the intersection of 
 the component of $0$ of their formal domains, 
 then they also agree on any matrices where they are both defined.
 As noted earlier from 
  \cite{HMV}), 
  $s$ and $r$ being evaluation equivalent 
 (as rational expressions in symmetric variables)
  in some neighborhood of $0$ 
 implies they are series equivalent.
  Hence their power series about $0$ 
 in the $z_j$ variables 
 are identical.   Thus $s(Z)$ equals $r(Z)$ whenever evaluated at 
 any tuple $Z$ of (not necessarily symmetric) square matrices  
in a free neighborhood of $0$.

 Let $t=r-s$.  Next we show 
 that $ t(Z)=0$  on its entire formal domain
by  proving this claim for every matrix dimension $n$ separately. Let 
 $Z$ be a $g$-tuple of generic matrices of size $n$,
 that is, the $k^{th}$ entry is a matrix with entries which are the 
 commuting variables  $Z^k_{ij}$.
Then
$t(Z)=N/\delta$ where the numerator $N$ is a matrix
polynomial and the denominator is a scalar polynomial
in the variables  $\zeta:= \{Z^k_{ij}: \ all \ k,i,j \}$.
By assumption, $\delta$ is nonzero at some neighborhood 
of zero and $N$ is zero in  some
neighborhood of zero. 
It follows that $N(\zeta)$ is identically zero.

Now that we have extended $r$ 
uniquely to a free
rational function $\breve  r,$
Theorem 3.1  \cite{KvV} interpreted in our context  
implies that 
there are no hidden singularities in the ``free" $\domalg$.
 Since our symmetric variable  $\domalg$
is contained in ``free" $\domalg$, we conclude there are 
no hidden singularities in $\domalg$.
 \qed

%

\section{LMI representations}
\label{sec:repCor}
  In this section we first prove Theorem  \ref{cor:canthide}
  by showing that in each case the hypotheses imply that
  a hidden singularity is well hidden and then applying Theorem
 \ref{thm:main}.  
 The LMI representation 
  results of Theorem \ref{cor:main} are then shown to be consequences
  of Theorem \ref{cor:canthide} and the main result of \cite{HM}.

\begin{proposition}
\label{lem:unwell} 
 Suppose $r$ is a minimal symmetric descriptor realization.
 \begin{itemize}
  \item[(lim)]
  If for each $n$ the set $\fP_r(n)$ is convex and if $\chi \in\fP_r(N)$ is a hidden singularity,
 then $\chi$ is a well hidden singularity.
\item[(alg)]
 If $\chi\in\domalg(r,N)$ is a hidden singularity, then $\chi$ is a well hidden singularity. 
\end{itemize}
\end{proposition}

\begin{proof}
 The argument at the outset of the proof of Lemma \ref{lem:Xdirection} 
 shows that for $t\ne 1$ sufficiently close to $1$
 the matrix $J-L_A(t\chi)$ is invertible.

  To prove item (lim), suppose $\chi\in\fP_r(N)$. Thus
\[
 \lim_{t\to 1} r(t\chi) = r(\chi)\succ 0
\]
 from which it follows that $r(t\chi)\succ 0$ for $t$ close to $1$. 

  The assumption that $r(0)\succ 0$ implies, given $K\in\smatmg$, 
  there exists an $\eta>0$ such that if $|\rho|<\eta,$ then
 $r(\rho K)\succ 0$.  It follows that, for $\delta>0$ sufficiently small,
  both $(1+\delta)\chi\oplus \rho K$ and $(1-\delta)\chi \oplus \rho K$
 are in $\fP_r(N+m)$. Hence, by the convexity assumption, so is
  the average, $\chi \oplus \rho K$.  Hence $\chi$
  is a well hidden singularity. 

  To prove item (alg), 
 suppose $\chi$ is a hidden
  singularity of $r$ and $\chi\in\domalg(r,N)$.  In this case
 there exists a rational expression $r_*$ which is equivalent to
 $r$ such that $\chi$ is in the formal domain of
 $r_*$.  From the definitions, the formal domain of
 $r_*$ also contains a free neighborhood of $0$ and hence, given
  $K\in\smatmg,$ there is an $\eta>0$ such that if 
  $|\rho|<\eta$, then $\chi \oplus \rho K$ is also in 
  the  formal domain of $r_*$. 
  An application of  Lemma \ref{lem:limits}
  implies that 
  $\chi \oplus \rho K$, is a hidden 
  singularity of $r$. 
  Thus $\chi$ is a well hidden singularity of $r$. 
\end{proof}

\begin{corollary}
 \label{cor:unwellmain}
  Suppose $r$ is a minimal symmetric descriptor realization.
 \begin{itemize}
  \item[(lim)]
  If for each $n$ the set $\fP_r(n)$ is convex, then $\fP_r$ contains no hidden singularities.
\item[(alg)]
  The sets $\domalg(r,n)$ contain no hidden singularities.
\end{itemize}
\end{corollary}

\begin{proof}
  By Theorem \ref{thm:main} any hidden singularities in the sets $\fP_r(n)$ and 
  $\domalg(r,n)$ are well hidden. 
  By Proposition \ref{lem:unwell} the sets $\fP_r(n)$ and $\domalg(r,n)$ contain
  no well hidden singularities. 
\end{proof}



Given the descriptor realization $r$ and a positive integer $n$, let
\[
 \mathfrak{D}(r,n) = \{ X \in \smatng \;: \ X\in\mfI_{J-L_A(x)}(n)
   \mbox{ and }  r(X) \mbox{ is invertible}\}.
\]
 For  $X\in \mathfrak{D}(r,n)$ define 
\[
   s(X) = r(X)^{-1}.
\]
  Thus $s$ is the inverse of the rational expression $r$ and the sequence
  $(\mathfrak{D}(r,n))_n$ is the formal domain of $s$. 
  It is well known  that there is a minimal descriptor realization (evaluation) equivalent 
  to  $s$; i.e., there is a
  minimal symmetric descriptor realization  
\beq
\label{eq:rrep}
  \tr(x) = \tD +\tC^T(\tJ-L_{\tA}(x))^{-1} \tC
\eeq
  such that $\tr(X)= r(X)^{-1}$ for each $n$ and $X$ in the 
 open dense set  $\mathfrak{D}(r,n)\cap \mfI_{\tJ-L_{\tA}(x)}(n)$ 
 (c.f. \cite{HMV} Lemma 4.1).
 Recall 
$\cD_r(n)$ denotes the principal
 component  of the set
$  \{X\in\domalg(r,n): r(X)\succ 0\}
$ and $\cD_r=(\cD_r(n))_n$.

\begin{lemma}
 \label{lem:boundary}
  Suppose $r$ is a minimal symmetric descriptor realization.
\begin{itemize}
 \item[(lim)]
  \label{it:boundary-lim}
  If each $\fP_r(n)$ is convex, then  $\fP_r = \fP_{\tr}$.
   If $\chi$ is in the boundary of $\fP_r(n)$, then either
  $r$ has a singularity at $\chi$ or $\tr$ has a singularity at $\chi$.
  In particular,  $J-L_A(\chi)$ is singular or $\tJ-L_{\tA}(\chi)$
  is singular. 
  \item[(alg)]
  \label{it:boundary-alg}
   Likewise, for each $n$,
\begin{equation}
 \label{eq:domalgpos}
  \{X\in\domalg(r,n): r(X)\succ 0\} = \{X\in\domalg(\tr,n): \tr(X)\succ 0\}.
\end{equation}
  In particular
  $\cD_r=\cD_{\tr}.$  Further if $\chi$ is in the boundary
  of $\cD_r$ then either $J-L_A(\chi)$ or $\tJ-L_{\tA}(\chi)$ is singular. 
\end{itemize}
\end{lemma}

\begin{proof}
  To prove item (lim) suppose $\chi\in\fP_r(n)$.  
 From Corollary \ref{cor:unwellmain} we have
  $\chi\in\mfI_{J-L_A(x)}(n)$.  
  If  $\chi\in\mfI_{\tJ-L_{\tA}(x)}$, then 
  $\tr(\chi) = r(\chi)^{-1}\succ 0$ and thus $\chi\in\fP_{\tr}(n)$.   
  If instead $\chi \notin \mfI_{\tJ-L_{\tA}(x)}$, then, taking
  the limit through  $X\in\mfI_{\tJ-L_{\tA}(x)}$ and
  using the fact that, for $X$ near $\chi$, both $X\in\mfI_{J-L_A(x)}$ 
  and $r(X)\succ 0$ gives,
\[
 \lim_{X\to \chi} \tr(X) = \lim_{X\to \chi} r(X)^{-1} =r(\chi)^{-1} \succ 0.
\]
  It follows that $\chi$ is a hidden singularity 
  of $\tr$ and $\tr(\chi)$ (defined by this limit) is
  positive definite. Hence $\chi\in\fP_{\tr}(n)$  
   and thus $\fP_r(n) \subset \fP_{\tr}(n)$. 


 To prove the reverse inclusion, 
 suppose $\chi\in\fP_{\tr}(n)$.  
  If $\chi\in\mfI_{J-L_{A}(x)} \cap \mfI_{\tJ-L_{\tA}(x)}$, then 
  $r(\chi) = \tr(\chi)^{-1}\succ 0$ and thus $\chi \in\fP_r(n)$. 
  If instead $\chi\in \mfI_{\tJ-L_{\tA}}$, but $\chi \notin \mfI_{J-L_{A}(x)}$, then, taking
  the limit through  $X$ in the (dense) set $\mfI_{J-L_{A}(x)}$ and
  using the fact that $X\in\mfI_{\tJ-L_{\tA}(x)}$ 
  and $\tr(X)\succ 0$ for $X$ near $\chi$, 
\[
 \lim_{X\to \chi} r(X) = \lim_{X\to \chi} \tr(X)^{-1} =\tr(\chi)^{-1} \succ 0.
\]
  It follows that $\chi$ is a hidden singularity 
  of $r$ and $r(\chi)$ (defined by this limit) is
  positive definite.  Thus, if $\chi\in \mfI_{\tJ-L_{\tA}(x)}$,
   then $\chi\in \fP_r(n)$. 

  Next suppose $\chi\notin\mfI_{\tJ-L_{\tA}(x)}$ and 
  for notational ease, let $\mfI= \mfI_{J-L_A(x)}\cap \mfI_{\tJ-L_{\tA}(x)}$. 
  Note that $\mfI$ is open and dense. 
  In this case,
\[
  \lim_{\mfI\ni X\to \chi} \tr(X) = \tr(\chi)\succ 0.
\]
 Hence, 
\begin{equation}
 \label{eq:hardinclusion1}
  \lim_{\mfI\ni X\to \chi} r(X) = \lim_{\mfI\ni X\to \chi} \tr(X)^{-1} 
  = \tr(\chi)^{-1} \succ 0.
\end{equation}
 
  Letting $L=\tr(\chi)^{-1}$, Equation \eqref{eq:hardinclusion1} 
  implies, given $\epsilon>0$ there is a $\delta$ so that if
  $\|X-\chi\| <\delta$ and $X\in \mfI$, then $\tr(X)$ is invertible and 
  $\|r(X)-L\|<\epsilon$.  Now suppose only that $X\in \mfI_{J-L_A(x)}$
  and $\|X-\chi\|<\delta$.  In this case, using continuity of 
  $r$ at $X$ and the fact that $\mfI$ is open and dense, 
  there is an $Y\in \mfI$ such that $\|Y-\chi\|<\delta$
  and $\|r(X)-r(Y)\| <\epsilon$. 
   Hence, 
\[
  \|r(X)-L\|<2\epsilon.
\]
  It follows that
\[
  \lim_{\mfI_{J-L_A(x)} \ni X \to \chi} r(X) = L.
\]
   As also $L\succ 0$, it follows that 
  $\chi\in\fP_{r}(n)$  
   and thus $\fP_{\tr}(n) \subset \fP_{r}(n)$.

 To complete the proof of item (lim),
 suppose $\chi$ is in the boundary of $\fP_r$.  In this case,
 either 
\[
 \lim_{X\to\chi} r(X)
\]
 exists and is both positive
 semidefinite and singular, or 
 the limit fails to exist (and necessarily $X\not\in\mfI_{J-L_A(x)}$).
 In this second case, $r$ has a singularity at $\chi$.  In the first case,
 $\tr$ must have a singularity at $\chi$. 
 In particular, either $J-L_A(\chi)$ is singular, or $\tJ-L_{\tA}(\chi)$ is singular. 

  Now we turn to the proof of item (alg). 
  First suppose $X\in\domalg(r,n)$ and $r(X)\succ 0$.
   By Corollary \ref{cor:unwellmain} we see $X$
   is in the formal domain of $r$.  It follows that
   the rational expression $r^{-1}$ is defined
   at $X$ and hence, by Corollary \ref{cor:unwellmain} (alg), $X\in\domfor(\tr,n)$.
   Moreover, $\tr(X)\succ 0$ since
    $r^{-1}(X)\succ 0$.  Hence, $\cD_r\subset \cD_{\tr}$. 

  To establish the reverse inclusion in Equation \eqref{eq:domalgpos},
  fix $X\in\cD_{\tr}(n)$. 
  Let $\hr$ denote the rational expression $(\tr)^{-1}$.  
  In particular, $\hr$ is equivalent to $r$. 
  By Corollary \ref{cor:unwellmain}
  applied to $\tr$, the tuple $X$ is in fact in 
  the formal domain of $\tr.$ Since also $X\in\cD_{\tr}(n)$
  it follows that  $\tr(X)\succ 0$
  (and is thus invertible). Thus
   $X$ is in the formal domain of $\hr$. 
   Hence $X$ is in the algebraic domain of $r$ and consequently,
  by Corollary \ref{cor:unwellmain}, $X$ is in the formal
  domain of $r$ (equivalently $J-L_A(X)$ is invertible). 
  Further, $r(X)=\hr(X)=\tr(X)^{-1} \succ 0$ and the reverse
  inclusion, $\cD_{\tr}\subset \cD_r$,  in Equation \eqref{eq:domalgpos} is established.

  To complete the proof of item (alg),
  suppose $\chi$ is in the boundary of $\cD_r(n)$. 
  If $\chi\notin\domalg(r,n),$ then 
  $J-L_A(\chi)$ is not invertible. Thus, it may be assumed that
  $\chi\in\domalg(r,n),$
  but $r(\chi)$ is both positive semidefinite and singular. 
  If $\chi\in \domalg(\tr,n)$, then
\[
   \lim_{X\to \chi} r(X)^{-1} = \lim_{X\to \chi} \tr(X), 
\]
  where the limit is taken through $X\in \mfI_{r},$
  giving the contradiction that $r(\chi)$ is invertible. 
  Hence $\tJ-L_{\tA}(\chi)$ fails to be invertible.
\end{proof}

\begin{proof}[Proof of Theorem \ref{cor:main}]
  To prove item (lim), suppose $\fP_r$ is convex.  
 
 Let $\mfI_P^\circ$ denote the component of $0$ of the invertibility set 
  of the affine linear pencil
\[
 P(x)=J\oplus \tJ - L_{A\oplus \tA}(x) = [J-L_A(x)]\oplus [\tJ-L_{\tA}(x)].
\]
  Let $X\in\smatng$
  be given.  If $X\in\fP_r$, then for each $0\le t\le 1$
  we have $tX\in\fP_r$ by convexity. 
  By Corollary \ref{cor:unwellmain}
  $J-L_A(tX)$ is invertible
  for each $t$, since $\fP_r$ contains
  no hidden singularities of $r.$ 
  By Lemma \ref{lem:boundary}, $\fP_r=\fP_{\tr}$ and again
  an application Corollary \ref{cor:unwellmain}
  implies that $\tJ-L_{\tA}(tX)$
  is invertible for $0\le t\le 1$.  It follows that
  $X$ is in the component of zero of the invertibility set of
  $\mfI_P$ and hence   $\fP_r \subset \mfI_P^\circ$. 

  Now suppose $X$ is in the component of zero of the invertibility set
  of $P$.  Choose any path $F(t)$ connecting zero to $X$ which
  lies entirely in the component of zero of $\mfI_P$. 
  If this path does not lie  entirely in $\fP_r=\fP_{\tr}$,
  then there is a first point $t_0$ such that $F(t_0)$
  is in the boundary of $\fP_r$. 
  Hence, by Lemma \ref{lem:boundary}, $P(F(t_0))$ must
  be singular.  But then $F(t_0)$ is not in 
  the invertibility set of $P$ and the
  path $F$ does not lie entirely in the component
  of zero of the invertibility set of $P$,
  a contradiction.  Hence $X\in\fP_r$. Thus
  $\mfI_P^\circ \subset \fP_r$ and the equality
 $\mfI_P^\circ=\fP_r$ is proved.

%

We have established
   $\fP_r$ is the component of zero of the invertibility set of
  the symmetric matrix valued polynomial $P$. Assuming also
  that $\fP_r$ is bounded, the main result of
  \cite{HM} says that  $\mfI_P^\circ$ has an LMI representation; i.e., there exists $\bA$ such that
 $X$ is in the component of zero of the set where $P$ is invertible if
 and only if $I-L_{\bA}(X)\succ 0$.

  The proofs of items (alg) and (lim) are similar
  as we see now.  
  Suppose $\cD_r$ is convex.  If $X\in\cD_r$, then for each $0\le t\le 1$
  we have $tX\in\cD_r$ by convexity.  By
   Corollary \ref{cor:unwellmain} the matrix
  $J-L_A(tX)$ is invertible
  for each $t$, since $\cD_r$ contains
  no hidden singularities of $r.$ 
  By Lemma \ref{lem:boundary}(alg), $\cD_r=\cD_{\tr}$ and again
  an application of Corollary \ref{cor:unwellmain}
  implies that $\tJ-L_{\tA}(tX)$
  is invertible for $0\le t\le 1$.  It follows that
  $X$ is in the component of zero of the invertibility set of
  the pencil $P$.
  Hence $\cD_r \subset \mfI_P^\circ$. 

  Now suppose $X$ is in the component of zero of the invertibility set
  of $P$.  Choose any path $F(t)$ connecting $0$ to $X$ which
  lies entirely in the component of zero of $P$. 
  If this path does not lie  entirely in $\cD_r=\cD_{\tr}$,
  then there is a first point $t_0$ such that $F(t_0)$
  is in the boundary of $\cD_r$. Hence, by
  Lemma \ref{lem:boundary} (alg), $P(F(t_0))$ must
  be singular.  But then $F(t_0)$ is not in 
  the invertibility set of $P$ and the
  path $F$ does not lie entirely in the component
  of zero of the invertibility set of $P$,
  a contradiction.  Hence $X\in\cD_r$. Thus
  $\mfI_p^\circ\subset \cD_r$ and the equality
 $\mfI_p^\circ=\cD_r$ is proved.

We have established that
  $\cD_r$ is the component of zero of the invertibility set of
  the symmetric matrix valued polynomial $P$. Assuming also
  that $\cD_r$ is bounded, the main result of
  \cite{HM} says that  $\mfI_P^\circ$ has an LMI representation; i.e., there exists $\bA$ such that
 $X$ is in the component of zero of the set where $P$ is invertible if
 and only if $I-L_{\bA}(X)\succ 0$. 
\end{proof}

\end{document}